\documentclass[usemulticol,english]{ccdconf}
\usepackage[utf8]{inputenc}
\pdfoutput=1

\usepackage{amsmath}
\usepackage{amssymb}
\usepackage{amsthm}
\usepackage{graphicx}

\theoremstyle{definition}
\newtheorem{theorem}{Theorem}
\newtheorem{corollary}{Corollary}
\newtheorem{definition}{Definition}

\newtheorem{example}{Example}

\usepackage[maxbibnames=6,doi=false,isbn=false,url=false,eprint=false,style=ieee,citestyle=numeric-comp]{biblatex}
\AtEveryBibitem{\clearfield{issue}}

\renewbibmacro{in:}{}

\DeclareFieldFormat[article, inbook, incollection, inproceedings, misc, thesis, unpublished]{title}{#1}
\DeclareFieldFormat{journaltitle}{#1\isdot}
\DeclareFieldFormat[article, inbook, incollection, inproceedings, misc, thesis, unpublished]{volume}{Vol.#1}

\usepackage{diagbox}
\usepackage{blkarray}

\usepackage{titlesec}
\titlespacing*{\section}{0pt}{3ex}{1ex}

\begin{document}

\title{Payoff Control in Repeated Games}
    
\author{{Renfei Tan}\aref{pku}, Qi Su\aref{upenn}, Bin Wu\aref{bupt,corr}, Long Wang\aref{pku,corr}}

    \affiliation[pku]{Center for Systems and Control, College of Engineering, Peking University, Beijing, 100871, China
            \email{tanrenfei@pku.edu.cn, longwang@pku.edu.cn}}
    \affiliation[upenn]{Department of Mathematics and Department of Biology, University of Pennsylvania, PA, USA
            \email{qisu1991@sas.upenn.edu}}
    \affiliation[bupt]{School of Science, Beijing University of Posts and Telecommunications, Beijing 100876, China
            \email{bin.wu@bupt.edu.cn}}
    \affiliation[corr]{Corresponding Authors.}

\maketitle

\begin{abstract}
    Evolutionary game theory is a powerful mathematical framework to study how intelligent individuals adjust their strategies in collective interactions. 
    It has been widely believed that it is impossible to unilaterally control players' payoffs in games, since payoffs are jointly determined by all players. 
    Until recently, a class of so-called zero-determinant strategies are revealed, which enables a player to make a unilateral payoff control over her partners in two-action repeated games with a constant continuation probability.
    The existing methods, however, lead to the curse of dimensionality when the complexity of games increases. 
    In this paper, we propose a new mathematical framework to study ruling strategies (with which a player unilaterally makes a linear relation rule on players' payoffs) in repeated games with an arbitrary number of actions or players, and arbitrary continuation probability. 
    We establish an existence theorem of ruling strategies and develop an algorithm to find them.
    In particular, we prove that strict Markov ruling strategy exists only if either the repeated game proceeds for an infinite number of rounds, or every round is repeated with the same probability. 
    The proposed mathematical framework also enables the search of collaborative ruling strategies for an alliance to control outsiders. 
    Our method provides novel theoretical insights into payoff control in complex repeated games, which overcomes the curse of dimensionality.
\end{abstract}


\section{INTRODUCTION}

Game theory was first developed in economics to describe how rational individuals make decisions when facing conflicts \cite{Neumann1944,Nash1950}. It has been widely used in computer science, physics, management science, etc.
When playing games, one's payoff depends on not only her own action but also her opponent's.
Typically, in the symmetric two-player two-action game, each player can choose either cooperation or defection.
The mutual cooperation brings each player a reward $R$ and the mutual defection leads to punishment $P$; the unilateral cooperation gives the cooperator a sucker payoff $S$ and the defector a temptation $T$. Different ranks of $R,S,T,P$ lead to different conflicts.

Evolutionary game theory enriches the classical game theory by introducing adaptive action learning \cite{Smith1973,Smith1974}.  
It is a powerful tool to study the traits' evolution and stabilization in systems consisting of interacting individuals, ranging from engineering, economics to sociology. 
A representative example is the evolution of large-scale cooperation, a prevailing phenomenon in various complex systems \cite{Nowak2006}. 
Prisoner's dilemma, one of the classical two-player two-action games, well captures the conflicts between individuals and groups.
It describes that one's choosing defection is always better than cooperation in terms of her own benefit, which then leads to mutual defection, a worse outcome for the group interests than the mutual cooperation.
The last three decades have seen numerous efforts on this topic with the aid of evolutionary game theory \cite{Wu2010,Su2019,Li2020}. 

Direct reciprocity is one of the major mechanisms responsible for the evolution of cooperation.
It tells that interacting individuals play games repeatedly and adjust their actions from round to round.
Repeated games, therefore, provide the flexibility for players to seek a good strategy to dominate the opponents in the long-term interactions \cite{HilbeArne2015,Stewart2012}.
So far, there are a few good strategies, such as Tit-for-tat (TFT, taking the opponent's action used last round) \cite{Axelrod1981}, and win-stay, lose-shift (WSLS, taking the same action if performed well last round and switching otherwise) \cite{Nowak1993}.

A strategy that exerts payoff control is presumably a good strategy, since it controls payoffs and guarantees an advantage.
But it could be difficult to find such control strategies.
This is because, on the one hand, the payoff is determined by the action profiles of all players, on the other hand, each player is only able to decide her own action. 
Recently a class of strategies called \textit{zero-determinant (ZD) strategies} \cite{Press2012} have been discovered.
ZD strategies enable a player to \textit{unilaterally} enforce a linear payoff relation between the two players in repeated prisoner's dilemma. 
These strategies are able to pin the opponent's payoff or guarantee an above-average payoff for the controller, whatever strategies her opponent uses.
The discovery of ZD strategies explicitly indicates that payoff control exists in repeated games.
Such a surprising fact has sparked a surge of interests in searching for such control strategies in other repeated games: ZD strategies have been found in repeated symmetric public goods games, games with continuous action sets, and finitely repeated prisoner's dilemma \cite{HilbeWu2014,Pan2015,McAvoy2016}.
The evolutionary performance of ZD strategies in two-player and multiplayer social dilemma is discussed in \cite{Adami2013,Stewart2013,HilbeWu2015,Hilbe2013}.
In addition, behavioral experiments demonstrates the existence of ZD-like strategies in human societies \cite{HilbeTor2014}. 

The previous framework confronts a mathematical problem when applied to other repeated games beyond the prisoner's dilemma. 
Searching for zero-determinant strategies, as the name indicates, involves operating on a matrix to make its determinant vanish. To verify a ZD strategy in an \(n\)-player \(m\)-action repeated game, one needs to construct an \(m^n\times m^n\) transition matrix with \(m^{2n}\) entries. The matrix grows exponentially with the number of players, leading to the curse of dimensionality.
This problem hinders the use of the determinant method in complex repeated games.

In this paper, we focus on the algebraic structure of the strategy space to provide a new mathematical framework. 
To this end, we formally define ruling strategy, with which a player is able to unilaterally make a payoff rule between players. 
The algebraic perspective greatly simplifies the process of finding ruling strategies in more complicated games. 
In the same \(n\)-player \(m\)-action game, with our method one only needs to solve a linear equation with \(n+r+1\) variables (we later show that for Markov ruling strategies \(r=m-1\)), thus removing the curse.
It also enables the search for \textit{collaborative} ruling strategies, by which multiple rulers work collectively to rule outsiders.
We show that working collectively improves the ability to control payoffs, which highlights the synergy effects in alliances \cite{Wu2016}.
We develop a novel searching algorithm applicable to \textit{any} repeated games with finite players and finite action sets. 
Additionally, we show that strict Markov ruling strategies exist only if the game either has infinite expected rounds, or every round is repeated with the same probability.

The paper is organized as follows. 
In Section \ref{Sec_Preliminaries}, preliminaries are provided concerning the repeated game model. 
In Section \ref{Sec_ConstraintStrategy}, we formally define the ruling strategy and collaborative ruling strategy. 
We discuss their control mechanisms, based on which a searching algorithm for these strategies is proposed. 
In Section \ref{Sec_method}, we establish a theorem to seek Markov ruling strategies.
In Section \ref{Sec_Discussion} we discuss further applications of our algorithm and end up with conclusions.


\section{PRELIMINARIES}
\label{Sec_Preliminaries}

In this section, we formulate the repeated game model.
In our model, repeated games consist of two parts, the game being played each round (called \textit{base game}), and the probability to play another round (called \textit{continuation probability}). 
Then we introduce the \textit{history} and a \textit{strategy} in a repeated game.
Particularly, we define the \textit{effective payoff} for a player, i.e., the payoff controlled by ruling strategies.

\subsection{Games in strategic forms}
\label{Subsec_Strategicform}

We first define an $n$-player base game \cite{Maschler2013}, which is played by $n$ players, labelled by the set \(N=\{1,2,\dotsb,n\}\).
Each player $i$ has $m_i$ action options, labelled by the set \(A_i=\{a_i^1,a_i^2,\dotsb,a_i^{m_i}\}\). 
Note that for different players, their action sets do not necessarily overlap. 
In each base game, every player chooses an action.
Player $i$'s payoff $u_i(\mathbf{a})$, is determined by the formed \textit{action profile} $\mathbf{a}=(a_1,a_2,\dotsb,a_n)$.
The set of all action profiles is \(A=A_1\times A_2\times \dotsb \times A_n\) and the set size is $\Pi_{i=1}^n m_i$.
Let $\Pi_{i=1}^n m_i$-entry tuple $\mathbf{u}_i$ denote player $i$'s payoff in all action profiles. $\mathbf{u}_i$ is the \textit{payoff vector} of \(i\).

More generally, a player is expected to take \textit{mixed actions} rather than pure actions. A mixed action records the probability to take a certain action. For example, if player $i$ uses mixed action $p_i$, she takes action $a_i$ with probability $p_i(a_i)$.
The set of mixed actions of player $i$ is defined as:
    \begin{equation}
        \Sigma_i = \bigg\{p_i : A_i \to [0,1]\Big\vert \sum_{a_i\in A_i}p_i(a_i)=1\bigg\}.
    \end{equation}
When players take mixed actions $p_1,p_2,\dotsb,p_n$, the interacting scenario corresponds to a probability distribution over action profile set $A$. The probability of an action profile is the product of probability that every player takes the corresponding action. 
Let $\mathbf{v}$ denote the probability distribution of all action profiles and $v(\mathbf{a})$ the probability of action profile $\mathbf{a}$. 
Then player $i$'s expected payoff $u_i$ is an inner product between her payoff vector and the probability distribution:
\begin{equation}
    u_i(\mathbf{v})=\sum_{\mathbf{a}\in A}\mathbf{u}_i(\mathbf{a})v(\mathbf{a}).
\end{equation}  
For simplicity, we denote it as
    \begin{equation}
        u_i(\mathbf{v}) = \langle \mathbf{u}_i,\mathbf{v} \rangle.
    \end{equation}

\begin{example}[Prisoner's dilemma and mixed actions]
We consider the prisoner's dilemma, a kind of two-player two-action game.
These payoffs satisfy \(T > R > P > S\) and \(2R > T+S\).
On the one hand, \(2R > T+S\) and \(R>P\) assure that cooperation is beneficial for the group. On the other hand, \(T>R\) and \(P>S\) assure that rational individuals would defect regardless of the opponent's action. Thus, it leads to a game with a conflict between individual and group benefit.
The quantities $\mathbf{a}, \mathbf{u}_i$, and $\mathbf{v}$ are 
    \begin{equation}
        \begin{blockarray}{ccccc}
            \mathbf{a}= & \mbox{CC} & \mbox{CD} & \mbox{DC} & \mbox{DD}\\
            \begin{block}{c[cccc]}
                \mathbf{u}_1= & R & S & T & P\\
            \end{block}
            \begin{block}{c[cccc]}
                \mathbf{u}_2= & R & T & S & P\\
            \end{block}
            \begin{block}{c[cccc]}
                \mathbf{v}= & v(\small{\mbox{CC}}) & v(\small{\mbox{CD}}) & v(\small{\mbox{DC}}) & v(\small{\mbox{DD}})\\
            \end{block}
        \end{blockarray},
    \end{equation}
where in $\mathbf{a}$, the first element is player $1$'s action and the second element player $2$'s action.
Players' payoffs are given by \(u_1(\mathbf{v})=\langle \mathbf{u}_1,\mathbf{v} \rangle\) and \(u_2(\mathbf{v}) = \langle \mathbf{u}_2,\mathbf{v} \rangle\).
\end{example}

\subsection{Repeated games}
\label{Subsec_RepeatedGames}

The most intensively studied repeated games are games that repeat infinitely and games that repeat for a finite amount of times.
We focus on more general repeated games which we name as \textit{generalized repeated games}.
In a generalized repeated game, after finishing each round, a time-dependent chance move decides whether the repeated game proceeds or not.
The probability to continue the game is called the continuation probability.
This model includes the two well-known types of repeated games, yet provides a variety of other repeated games.

\begin{definition}[Generalized repeated game]
    Let \(\Gamma = (N,(\Sigma_i)_{i\in N},(\mathbf{u}_i)_{i\in N})\) be an \(n\)-player game. A \textit{generalized repeated game} is
    \begin{equation}
        \Gamma_{GR}=\big(N,(\Sigma_i)_{i\in N},(\mathbf{u}_i)_{i\in N},c\big), 
    \end{equation}in which
    \begin{itemize}
        \item \(N,(\Sigma_i)_{i\in N},(\mathbf{u}_i)_{i\in N}\) are player set, mixed action sets and payoff vectors respectively, as defined in Section \ref{Subsec_Strategicform}.
        \item \(c:\mathbb{Z}^+\to [0,1]\) is a function mapping each positive integer \(t\) to a probability \(c(t)\). \(c(t)\) is the \textit{continuation probability} in round \(t\).
    \end{itemize}
\end{definition}

Players play the base game in the first round for sure. 
But whether or not the game proceeds depends on the continuation probability $c$.
Specifically, after finishing $t_{th}$ round, they play $(t+1)_{th}$ round game with probability $c(t)$ and stop otherwise.
The probability that the game proceeds at least $t$ rounds is
    \begin{equation}
        p(t)=
        \begin{cases}
            1 & t=1,\\
            c(1)c(2)\dotsb c(t-1) & t>1.\\
        \end{cases}
    \end{equation}
When $\forall t, c(t)=1$, players play infinite rounds of games.
Another example is $\forall t, c(t)=\delta$ ($\delta \in [0,1)$).
That is, after each round, the next round proceeds with a constant probability $\delta$, which is called $\delta$-repeated game.
$\delta=0$ recovers a one-shot game.

Compared with one-shot games, repeated games provide players with chances to reciprocate partners based on prior interactions. 
The information available to all players in round $t+1$ is the actions played in the first \(t\) rounds of the game.
Let \(h^t=(\mathbf{a}^1,\mathbf{a}^2,\dotsb,\mathbf{a}^t)\) denote a $t$-round interaction \textit{history}, where \(\mathbf{a}^\tau\) is the action profile in round \(\tau\).
Therefore, the set of $t$-round history \(H(t)\) is defined as:
    \begin{equation}
        H(t):=A^t=\underbrace{A\times A\times \dotsb \times A}_{t\; \mbox{\footnotesize{times}}}.
    \end{equation}
Particularly, $H(0)=\{\emptyset\}$ and $h^0=\emptyset$, which indicates that there is no history when game starts.

A \textit{strategy} is an action plan on what mixed action to play after every possible history, for player \(i\) that is a function \(s_i\) mapping each finite history to a mixed action:
    \begin{equation}
        s_i : \bigcup_{t=0}^{\infty} H(t) \to \Sigma_i.
    \end{equation}
Intuitively, with strategy \(s_i\) and history \(h^t\), the conditional probability that player \(i\) chooses action \(a_i\) is given by:
    \begin{equation}
        p_i(a_i \mid h^t)=[s_i(h^t)](a_i).
    \end{equation}

When players' strategies are given, the probability distribution of actions profiles at any round $t$, i.e. $\mathbf{v}^t$, can be step-by-step calculated. 
Thus the expected payoff for player $i$ at round $t$ \(u_i^t=\langle \mathbf{u}_i,\mathbf{v}^t \rangle\). The expected payoffs are then used to obtain a player's effective payoff, which is used to evaluate her overall performance.

\begin{definition}[Effective payoff]
    In a generalized repeated game, a player $i$'s \textit{effective payoff} is defined as:
    \begin{equation}
        \bar{u}_i=
        \lim_{t\to\infty}
        \frac{p(1)u^1_i + p(2)u^2_i + \dotsb + p(t)u^t_i}
        {p(1) + p(2) + \dotsb + p(t)}.
    \end{equation}
\end{definition}

The effective payoff is a weighted average of each round's expected payoff. 
In fact, the numerator is the sum of payoffs over all rounds, with the weights \(p(\tau)\) being the probabilities that the repeated game proceeds to round $\tau$. 
The denominator, for \(t\) approaching infinity, corresponds to the expected rounds played.
For simplicity, denote by
    \begin{equation}
        \bar{\mathbf{v}}(t) = \frac{p(1)\mathbf{v}^1 + p(2)\mathbf{v}^2 + \dotsb + p(t)\mathbf{v}^t} {p(1) + p(2) + \dotsb + p(t)}
    \end{equation}
the weighted average of distributions of the first $t$ rounds.
Then player \(i\)'s effective payoff satisfies: 
    \begin{equation}
        \bar{u}_i=
        \lim_{t\to\infty}\langle \mathbf{u}_i, \bar{\mathbf{v}}(t) \rangle.
    \end{equation}
Therefore, a player's effective payoff is jointly determined by strategies from all the players, i.e.
\begin{equation}
    \bar{u}_i = \bar{u}_i ( s_1 , s_2 , \dotsb , s_n),\;\;\;\;\forall i\in N.
\end{equation}

In the sequel, we refer to a player's \textit{effective payoff} as \textit{payoff}.

\section{RULING STRATEGY}
\label{Sec_ConstraintStrategy}

In this section we introduce \textit{ruling strategies} and \textit{collaborative ruling strategies}. 
We then define \textit{ruling vectors}. 
We prove that ruling strategies are closely associated with the linear space spanned by ruling vectors. 
Based on the theorems and discussions, we provide an algorithm to seek ruling strategies for a single player and collaborative ruling strategies for an alliance.

We first consider linear relations for payoffs.
A linear relation is an equation taking the form:
\begin{equation}
    \alpha_1 \bar{u}_1 + \alpha_2 \bar{u}_2 + \dotsb + \alpha_n \bar{u}_n + \gamma = 0.
\end{equation}
We define a linear payoff relation to be \textit{trivial} if it is always satisfied regardless of the strategies used by all the players.
For example, the linear relation \(0\bar{u}_1+0\bar{u}_2=0\), is satisfied in any two-player game. 
Another example is that in a two-player zero-sum game, the linear relation \(\bar{u}_1+\bar{u}_2=0\) is always satisfied. 
Ruling strategy enforces a non-trivial linear relation rule in the game.

\begin{definition}[Ruling strategy] 
    In an $n$-player repeated game, 
    a ruling strategy \(s_k^*\) (used by player $k$) is such a strategy with which regardless of strategies used by the rest, 
      \begin{itemize}
          \item the limit \(\bar{u}_i(\dotsb,s_k^*,\dotsb)\) exists for each player \(i\);
          \item there exist constants \(\alpha_1,\alpha_2,\dotsb,\alpha_n,\gamma\) which is unilaterally decided by player $k$, such that all players' payoffs have the non-trivial linear relation: $\alpha_1 \bar{u}_1 + \alpha_2 \bar{u}_2 + \dotsb + \alpha_n \bar{u}_n + \gamma = 0$.
      \end{itemize}
    \label{Def_Constraintstrategy}
\end{definition}

Intuitively, the player with a ruling strategy \textit{unilaterally} establishes a linear relation rule of payoffs among all players, whatever strategies the rest use.
This linear payoff rule is always satisfied and the control is exerted unilaterally by the focal player who uses a ruling strategy.

Sometimes, alliance rules and individuals do not.
Consider a 3-player one-shot voting game. 
Every player is both a voter and a candidate.
Each of them vote for a player, and the player with the most votes wins. 
If more than two players tie for first, no one wins. 
Each individual alone cannot control the outcome.
However, when two of the players form an alliance, they are able to communicate ahead and decide who is the winner.
If they focus their votes on a particular player, then she is guaranteed to win.
The alliance are able to control the voting game.
The same idea applies in generalized repeated games.
Multiple players can ally and collaboratively make a payoff relation rule, which cannot be made by a single player.

\begin{definition}[Collaborative ruling strategy]
    In an $n$-player repeated game, a collaborative ruling strategy set is such a set of strategies 
    \(\{s_{k_1}^*, s_{k_2}^* , \dotsb, s_{k_q}^*\}\) (used by $k_1, k_2, \cdots, k_q$), with which regardless of the strategies used by the rest, 
        \begin{itemize}
            \item the limit \(\bar{u}_i(\dotsb,\underbrace{s_{k_1}^*,s_{k_2}^*,\dotsb,s_{k_q}^*}_{k_1, k_2, \cdots, k_q\mbox{\footnotesize{'s strategies}}},\dotsb)\) exists for each player \(i\);
            \item there exist constants \(\alpha_1,\alpha_2,\dotsb,\alpha_n,\gamma\) which is unilaterally decided by player $k_1, k_2, \cdots, k_q$, such that all players' payoffs have the non-trivial linear relation: $\alpha_1 \bar{u}_1 + \alpha_2 \bar{u}_2 + \dotsb + \alpha_n \bar{u}_n + \gamma = 0$.
        \end{itemize}
\end{definition}

The payoff relation rule is a result of collaboratively decision making.
Note that even when an individual can never rule over the rest, a group of players can make such a payoff relation rule.
Also, a single strategy in the collaborative set may not be effective. 
It requires that every member of the alliance use the strategy accordingly to enforce the rule.

To control the payoffs, ruling strategies generate ruling vectors.
Ruling vectors are a class of vectors whose inner products with the weighted average of the distribution are always zero.

\begin{definition}[Ruling vector]
    In an $n$-player repeated game, for players \(k_1,k_2,\dotsb,k_q\) with strategies \(s_{k_1}^*,s_{k_2}^*,\dotsb,s_{k_q}^*\), a \textit{ruling vector} \(\tilde{\mathbf{u}}\) is such a vector that is unilaterally decided by \(s_{k_1}^*,s_{k_2}^*,\dotsb,s_{k_q}^*\)
    and that regardless of the strategies used by the rest, 
    the following equation always holds:
        \begin{equation}
            \lim_{t\to\infty}\langle \tilde{\mathbf{u}},
            \bar{\mathbf{v}}(t) \rangle=0.
            \label{Equ_PDvecs}
        \end{equation}
\end{definition}

Ruling vectors follow the superposition principle. 
That is, for two ruling vectors \(\tilde{\mathbf{u}}_1,\tilde{\mathbf{u}}_2\), their linear combination like \(d_1\tilde{\mathbf{u}}_1+d_2\tilde{\mathbf{u}}_2\) is a ruling vector as well. 
Therefore, for strategies \(s_{k_1}^*,s_{k_2}^*,\dotsb,s_{k_q}^*\), the set of all ruling vectors is a linear subspace. 
A basis is sufficient to describe the entire subspace. We term the subspace \textit{ruling space}.
An intuition for ruling space is that it is the kernel of the weighted average of the distribution \(\bar{\mathbf{v}}\). 
Note that ruling vectors, ruling spaces and their basis are determined by strategies \(s_{k_1}^*,s_{k_2}^*,\dotsb,s_{k_q}^*\) and can change as these strategies vary. 
The following theorems reveal that the existence of ruling vectors enables a linear payoff rule.

\begin{theorem}[Existence of ruling vectors]
    If \(s_{k_1}^*,s_{k_2}^*,\dotsb,s_{k_q}^*\) is a ruling strategy (are collaborative ruling strategies) and 
    \(\alpha_1,\alpha_2,\dotsb,\alpha_n,\gamma\) are the constants in the linear payoff relation, 
    vector \(\tilde{\mathbf{u}}=\alpha_1\mathbf{u}_1+\alpha_2\mathbf{u}_2+\dotsb+\alpha_n\mathbf{u}_n+\gamma\mathbf{1}\) is a ruling vector. \(\mathbf{1}\) is a $\Pi_{i=1}^n m_i$-entry vector with all entries being \(1\).
    \label{Thm_CStoPDvec}
\end{theorem}

This theorem holds because the equation of the linear payoff relation can be written as an inner product between linear combination and the weighted average of distribution:
    \begin{equation}
        \begin{split}
            &\alpha_1 \bar{u}_1 + \alpha_2 \bar{u}_2 + \dotsb + \alpha_n \bar{u}_n + \gamma = 0 \\
            \Leftrightarrow &\lim_{t\to \infty}\langle \alpha_1\mathbf{u}_1+\alpha_2\mathbf{u}_2+\dotsb+\alpha_n\mathbf{u}_n+\gamma\mathbf{1}, \bar{\mathbf{v}}(t) \rangle=0.\\
        \end{split}
    \end{equation}
\begin{theorem}[Equivalent condition for ruling strategy]
    Strategy(ies) \(s_{k_1}^*,s_{k_2}^*,\dotsb,s_{k_q}^*\) is a ruling strategy (are collaborative ruling strategies), if and only if,
        \begin{equation}
            \mbox{span}\{\mathbf{u}_1,\mathbf{u}_2,\dotsb,\mathbf{u}_n,\mathbf{1}\}\cap \mbox{span}\{\tilde{\mathbf{u}}_1,\dotsb,\tilde{\mathbf{u}}_r\}\neq\{\mathbf{0}\},
        \end{equation}
    in which \(\{\tilde{\mathbf{u}}_1,\dotsb,\tilde{\mathbf{u}}_r\}\) is a basis for the ruling space for respective strategy(s).
    \label{theorem:Judgment_of_CS}
\end{theorem}
\begin{proof}
    If the intersection is not a set with only zero vector, then any non-zero vector in the intersection is both a ruling vector and a linear conbination of payoff vectors:
    \begin{equation}
        \tilde{\mathbf{u}}=\alpha_1 \mathbf{u}_1 + \alpha_2 \mathbf{u}_2 + \dotsb + \alpha_n \mathbf{u}_n + \gamma \mathbf{1}.
    \end{equation}
    Therefore the same linear conbination of payoffs always vanishes.
\end{proof}

Theorem \ref{Thm_CStoPDvec} and \ref{theorem:Judgment_of_CS} reveal a fundamental relation between ruling vectors and a ruling strategy. 
The reason why ruling strategy rules is that its ruling space intersects with the linear span of payoff vectors.
Therefore, each vector in the intersection establishes a linear payoff rule.
This idea is explicitly illustrated in the following corollary.

\begin{corollary}
    \(s_{k_1}^*,s_{k_2}^*,\dotsb,s_{k_q}^*\) is a ruling strategy (are collaborative ruling strategies), if and only if, equation
    \begin{equation}
        \begin{bmatrix}
            \mathbf{u}_1 & \dotsb & \mathbf{u}_n & \mathbf{1}\\
        \end{bmatrix}
        \begin{bmatrix}
            \alpha_1 \\
            \vdots \\
            \alpha_n \\
            \gamma
        \end{bmatrix}
        = \begin{bmatrix}
            \tilde{\mathbf{u}}_1  & \dotsb & \tilde{\mathbf{u}}_r\\
        \end{bmatrix}
        \begin{bmatrix}
            y_1 \\
            \vdots \\
            y_r \\
        \end{bmatrix}
        \label{Eqt_FindCS}
    \end{equation}
    has non-zero solutions.
    \label{corollary:Equation_for_CS}
\end{corollary}

Theorem \ref{theorem:Judgment_of_CS} and Corollary \ref{corollary:Equation_for_CS} present the main idea of searching ruling strategies.
To find the collaborative ruling strategies for players \(k_1,k_2,\dotsb,k_q\), the algorithm works as follows: input the basis of ruling space into Eq.\ref{Eqt_FindCS}; solve the equation; 
find feasible strategies that satisfy the solution; end up with the constants \(\alpha_1,\dotsb,\alpha_n,\gamma\) from the solution.
This algorithm also verifies whether a certain linear payoff relation can be enforced by ruling strategies: input the constants in the relation into the equation. If the equation has a solution, the payoff relation is feasible.
However, we still need the expression of ruling vectors to complete the equation. 
We focus on solving this problem in the next section.

\section{MARKOV RULING STRATEGY}
\label{Sec_method}

In this section, we focus on Markov ruling strategies, with which only the interaction in the latest round is used to decide the action in the current round.
We then provide expressions for ruling vectors under two situations, and work out an example of ruling strategies in a two-player three-action game.
Another example is used to discuss the synergy effects when rulers ally.
This section ends up with an existence theorem for ruling vectors.
We prove that ruling vectors exist for strict Markov strategy if and only if either the repeated game proceeds for an infinite number of rounds, or every round is repeated with the same probability.

\begin{definition}[Markov strategy]
    In an $n$-player repeated game, strategy \(s_i\) for player \(i\) is a Markov strategy if actions based on $s_i$ and history \(h^{t}=(\mathbf{a}^1,\dotsb,\mathbf{a}^t)\) with \(t\neq 0\), satisfy
    \begin{equation}
        p_i(a_i^j \mid h^t)=p_i(a_i^j \mid \mathbf{a}^t).
    \end{equation}
\end{definition}

The definition of Markov strategies implies that the player has a one-step memory. Her behavior only depends on the action profile from the last round. Therefore, a Markov strategy can be represented by examining the behavior under length-zero and every length-one history (action profile) as inputs. Let $\Pi_{i=1}^n m_i$-entry tuple \(\mathbf{s}_{a_i^j}\) denote player \(i\)’s conditional probability to choose action \(a_i^j\) under different action profiles. Also, we define \(s_{a_i^j\mid 0}\) as the probability to choose action \(a_i^j\) in round \(1\) (with length-zero histories). We provide an example of Markov strategies in repeated prisoner's dilemma.

\begin{example}[Markov strategy in repeated prisoner's dilemma]
    In the two-player two-action game, a Markov strategy can be described as:
    \begin{equation}
        \begin{split}
            \mathbf{s}_C &= \begin{bmatrix}
                p_{C\mid CC} & p_{C\mid CD} & p_{C\mid DC} & p_{C\mid DD} \\
            \end{bmatrix},\\
            s_{C\mid 0}&=p_{C\mid 0},\\
        \end{split}
        \label{equation:Markov_strategy_described_by_cooperate_probability}
    \end{equation}
    where \(p_{C\mid \mathbf{a}}\) is the conditional probability to \textit{cooperate} in the next round given that the action profile $\mathbf{a}$ in the current round. 
    Analogously, based on the probability to \textit{defect}, a Markov strategy is given by
    \begin{equation}
        \begin{split}
            \mathbf{s}_D &= \begin{bmatrix}
                p_{D\mid CC} & p_{D\mid CD} & p_{D\mid DC} & p_{D\mid DD} \\
            \end{bmatrix},\\
            s_{D\mid 0}&=p_{D\mid 0}.\\
        \end{split}
        \label{equation:Markov_strategy_described_by_defect_probability}
    \end{equation}
    \label{example:Markov_strategy_in_IPD}
\end{example}

The following theorems provide ruling vectors for Markov strategies in any multi-player multi-action game. Whereas Akin's Lemma \cite{Akin2016}, which concentrates on a two-player two-action game.

\begin{theorem}[Ruling vectors for Markov strategy]
    In a generalized game, for any Markov strategy \(s_i\) and action \(a_i^j\), 
    \begin{itemize}
        \item if expected number of rounds is infinite, vector
        \begin{equation}
            \mathbf{s}_{a_i^j}-\mathbf{s}^{Rep}_{a_i^j}
            \label{Eqt_RVforIndividualInfinite}
        \end{equation}
    is a ruling vector;
        \item if \(\Gamma_{GR}\) is a \(\delta\)-repeated game, vector
        \begin{equation}
            \delta\mathbf{s}_{a_i^j}+(1-\delta)s_{a_i^j\mid 0}\mathbf{1}-\mathbf{s}^{Rep}_{a_i^j}
            \label{Eqt_RVforIndividualDelta}
        \end{equation}
    is a ruling vector.
    \end{itemize}
    
    In both cases, \(\mathbf{s}^{Rep}_{a_i^j}\) is an indicator vector with $\Pi_{i=1}^n m_i$ entries. Each entry corresponds to an action profile. The entry is $1$ when player $i$ uses action $a_i^j$ in the corresponding action profile. Otherwise the entry is $0$.
    \label{theorem:PDv_for_M-1strategy}
\end{theorem}

\begin{proof}
    The probability player \(i\) chooses action \(a_i^j\) in round \(t+1\) can be calculated from the distribution of round \(t\) and player \(i\)'s strategy \(\mathbf{s}_{a^j_i}\), and also from the distribution of round \(t+1\) and repeat strategy \(\mathbf{s}_{a^j_i}^{Rep}\).
    \begin{equation}
        \mbox{Prob.}(\mbox{``\(a_i^j\) in round \(t+1\)''})=  \langle \mathbf{s}_{a^j_i}^{Rep},\mathbf{v}^{t+1} \rangle = \langle \mathbf{s}_{a^j_i},\mathbf{v}^{t} \rangle.
    \end{equation}
For convenience we denote this probability by \(\mbox{P}_{a_i^j}^t\) and \(p(t)\) by \(p_t\).
Therefore, for games with infinite expected number of rounds, we have:
    \begin{equation}
        \begin{split}
            &\lim_{t\to\infty}\langle \mathbf{s}_{a^j_i}^{Rep}-\mathbf{s}_{a^j_i},
            \bar{\mathbf{v}}(t) \rangle\\
            =&\lim_{t\to\infty}\frac{p_t\mbox{P}_{a_i^j}^t+(p_{t-1}-p_t)\mbox{P}_{a_i^j}^t+\dotsb+(p_1-p_2)\mbox{P}_{a_i^j}^2-p_1\mbox{P}_{a_i^j}^1}{p_1 + p_2 + \dotsb + p_t}\\
            \leq&\lim_{t\to\infty}\frac{p_t\mbox{P}_{a_i^j}^t-p_t+p_1-p_1 \mbox{P}_{a_i^j}^1}{p_1 + p_2 + \dotsb + p_t}\\
            =&\;0.\\
        \end{split}
    \end{equation}
This is a use of sandwich theorem, and the inequality always holds because \(0\leq p_{t-1}-p_t\leq 1\) and \(0 \leq \mbox{P}_{a_i^j}^t \leq 1\).

For \(\delta\)-repeated games, similarly, we have:
    \begin{equation}
        \begin{split}
            &\lim_{t\to\infty}\langle \mathbf{s}^{Rep}_{a_i^j}-\delta\mathbf{s}_{a_i^j}-(1-\delta)s_{a_i^j\mid 0}\mathbf{1}, \bar{\mathbf{v}}(t) \rangle \\
            =&\lim_{t\to\infty} \frac{\mbox{P}_{a_i^j}^1-\delta^t \mbox{P}_{a_i^j}^t}{(1-\delta^t)/(1-\delta)}-(1-\delta)s_{a_i^j\mid 0}\\
            =&\lim_{t\to\infty} \frac{s_{a_i^j\mid 0}-\delta^t \mbox{P}_{a_i^j}^t}{(1-\delta^t)/(1-\delta)}-(1-\delta)s_{a_i^j\mid 0}\\
            =&\;0.\\
        \end{split}
    \end{equation}
The limit approaches 0 since \(\delta^t\to 0\). 
\end{proof}

\(\mathbf{s}^{Rep}_{a_i^j}\) can be viewed as a Markov strategy that \textit{repeats} whatever actions chosen in the previous round. Therefore, the conditional probability of action $a_i^j$ is $1$ when action profiles in the previous round contains the same action, and $0$ otherwise.

\begin{theorem}[Ruling vectors for Markov strategy set]
    In a generalized game, for every Markov strategy set \(\{s_{k_1},s_{k_2},\dotsb,s_{k_q}\}\), 
        \begin{itemize}
            \item if the expected rounds of games is infinite, vector
            \begin{equation}
                \mathbf{s}_{a_{k_1}^{j_1}\dotsb a_{k_q}^{j_q}}-\mathbf{s}^{Rep}_{a_{k_1}^{j_1}\dotsb a_{k_q}^{j_q}}
            \end{equation}
            is a ruling vector;
            \item if \(\Gamma_{GR}\) is a \(\delta\)-repeated game, vector
            \begin{equation}
                \delta \mathbf{s}_{a_{k_1}^{j_1}\dotsb a_{k_q}^{j_q}} + (1-\delta)s_{a_{k_1}^{j_1}\mid 0}\dotsb s_{a_{k_q}^{j_q}\mid 0}\mathbf{1}-\mathbf{s}^{Rep}_{a_{k_1}^{j_1}\dotsb a_{k_q}^{j_q}}
            \end{equation} 
            is a ruling vector.
        \end{itemize}
    In both cases, \(\mathbf{s}^{Rep}_{a_i^j}\) is an indicator vector with $\Pi_{i=1}^n m_i$ entries. Each entry corresponds to an action profile. The entry is $1$ when player $k_1,\dotsb,k_q$ uses action $a_{k_1}^{j_1},\dotsb,a_{k_q}^{j_q}$ respectively in the corresponding action profile. Otherwise the entry is $0$.
    And \(\mathbf{s}_{a_{k_1}^{j_1}\dotsb a_{k_q}^{j_q}}\) is also a $\Pi_{i=1}^n m_i$-entry vector with each entry corresponding to an action profile. The entry is the joint probability that actions $a_{k_1}^{j_1},\dotsb,a_{k_q}^{j_q}$ appear in the next round given the action profile previous round is \(\mathbf{a}\), namely
        \begin{equation}
            \mbox{Prob.}(a_{k_1}^{j_1},\dotsb,a_{k_q}^{j_q}\mid \mathbf{a})=p_{k_1}(a_{k_1}^{j_1}\mid \mathbf{a})\dotsb p_{k_q}(a_{k_q}^{j_q}\mid \mathbf{a}).
        \end{equation}
    \label{theorem:PDv_for_M-1Alliance}
\end{theorem}

\begin{proof}
    The method we used in the proof of Theorem \ref{theorem:PDv_for_M-1strategy} can be generalized to prove Theorem \ref{theorem:PDv_for_M-1Alliance}. For alliances, the joint probability that actions \(a_{k_1}^{j_1},\dotsb,a_{k_q}^{j_q}\) appear simultaneously in round $t$ satisfies:
    \begin{equation}
        \mbox{Prob.}(\mbox{``\(a_{k_1}^{j_1},\dotsb,a_{k_q}^{j_q}\)''})
        = \langle \mathbf{s}^{Rep}_{a_{k_1}^{j_1}\dotsb a_{k_q}^{j_q}},\mathbf{v}^{t+1} \rangle
        = \langle \mathbf{s}_{a_{k_q}^{j_q}\dotsb a_{k_q}^{j_q}},\mathbf{v}^t \rangle.
    \end{equation}
    The remainder is similar to the proof of Theorem \ref{theorem:PDv_for_M-1strategy}.
\end{proof}

Theorem \ref{theorem:PDv_for_M-1strategy} and \ref{theorem:PDv_for_M-1Alliance} enable us to search for ruling strategies and collaborative ruling strategies in the space of Markov strategies. For an \(n\)-player \(m\)-action game, Theorem \ref{theorem:PDv_for_M-1strategy} provides \(m\) ruling vectors, with \(m-1\) vectors being linearly independent, since the sum of all \(m\) ruling vectors equals zero. Analogously, if \(k\) players form an alliance, Theorem \ref{theorem:PDv_for_M-1Alliance} provides \(m^k\) ruling vectors, with \(m^k-1\) vectors being linearly independent. This indicates that working collectively expands exponentially the dimension of the ruling space, thus more linear relation rules are feasible.

The following examples illustrate how ruling strategies rule in different games. 
Example \ref{Emp_TwoPlayerThreeAction} demonstrates the algorithm in a two-player three-action infinitely repeated game. 
We display some unique ruling strategies, including a strategy that fixes the opponent's payoff and a strategy that guarantees equal payoffs.
The performances of these strategies, compared with the performance of a normal Markov strategy, are simulated and the results are shown in Figure \ref{Fig_TwoPlayerThreeAction}.

Example \ref{Exp_AllianceRules_IndividualsDont} focuses on the synergy effect of collaboration on a three-player two-action donor's game. 
By allying rulers are able to shelter their payoffs from the outsider's interruption, or they can reach out and control the outsider's payoff.
The results are shown in Figure \ref{Fig_AllianceRules}.

\begin{example}[Ruling strategy]
    Consider a two-player infinitely repeated donor's game. In each round, each player has three actions: (i) action \(C_1\): donates \(2\), for the opponent to obtain \(5\); (ii) action \(C_2\): donates \(1\), for the opponent to obtain \(3\); (iii) action \(D\) : donates nothing. There will be \(9\) feasible action profiles. The action profiles and corresponding payoffs under each profiles are given by:
        \begin{equation}
            \footnotesize{
            \begin{blockarray}{ccc}
                \mathbf{a} & \mathbf{u}_1 & \mathbf{u}_2\\
                \begin{block}{c[c][c]}
                    C_1C_1 & 3 & 3 \\
                    C_1C_2 & 1 & 4 \\
                    C_1D & -2 & 5 \\
                    C_2C_1 & 4 & 1 \\
                    C_2C_2 & 2 & 2 \\
                    C_2D & -1 & 3 \\
                    DC_1 & 5 & -2 \\
                    DC_2 & 3 & -1 \\
                    DD & 0 & 0 \\
                \end{block}
            \end{blockarray}
            }.
        \end{equation}
        Suppose that player \(1\) wants to control player \(2\)'s payoff. 
        It is an infinitely repeated game, therefore she first obtains her ruling vectors according to Eq.(\ref{Eqt_RVforIndividualInfinite}) in Theorem \ref{theorem:PDv_for_M-1strategy}:
            \begin{equation}
                \footnotesize{
                    \begin{blockarray}{ccc}
                        \tilde{\mathbf{u}}_1 & \tilde{\mathbf{u}}_2 & \tilde{\mathbf{u}}_3\\
                        \begin{block}{[ccc]}
                        p_{C_1\mid C_1C_1}-1 & p_{C_2\mid C_1C_1} & p_{D\mid C_1C_1}\\
                        p_{C_1\mid C_1C_2}-1 & p_{C_2\mid C_1C_2} & p_{D\mid C_1C_2}\\
                        p_{C_1\mid C_1D}-1 & p_{C_2\mid C_1D} & p_{D\mid C_1D}\\
                        p_{C_1\mid C_2C_1} & p_{C_2\mid C_2C_1}-1 & p_{D\mid C_2C_1}\\
                        p_{C_1\mid C_2C_2} & p_{C_2\mid C_2C_2}-1 & p_{D\mid C_2C_2}\\
                        p_{C_1\mid C_2D} & p_{C_2\mid C_2D}-1 & p_{D\mid C_2D}\\
                        p_{C_1\mid DC_1} & p_{C_2\mid DC_1}  & p_{D\mid DC_1}-1\\
                        p_{C_1\mid DC_2} & p_{C_2\mid DC_2} & p_{D\mid DC_2}-1\\
                        p_{C_1\mid DD} & p_{C_2\mid DD} & p_{D\mid DD}-1\\
                        \end{block}
                    \end{blockarray}
                }.
            \end{equation}
    According to Eq.(\ref{Eqt_FindCS}), we have:
        \begin{equation}
            \footnotesize{
            \begin{blockarray}{cccc}
                \mathbf{a} & \mathbf{u}_1 & \mathbf{u}_2 & \mathbf{1}\\
                \begin{block}{c[ccc]}
                    C_1C_1 & 3 & 3 & 1\\
                    C_1C_2 & 1 & 4 & 1\\
                    C_1D & -2 & 5 & 1\\
                    C_2C_1 & 4 & 1 & 1\\
                    C_2C_2 & 2 & 2 & 1\\
                    C_2D & -1 & 3 & 1\\
                    DC_1 & 5 & -2 & 1\\
                    DC_2 & 3 & -1 & 1\\
                    DD & 0 & 0 & 1\\
                \end{block}
            \end{blockarray}
            \begin{bmatrix}
                \alpha_1\\
                \alpha_2\\
                \gamma\\
            \end{bmatrix}
            =
            \begin{bmatrix}
                \tilde{\mathbf{u}}_1 & \tilde{\mathbf{u}}_2
            \end{bmatrix}
            \begin{bmatrix}
                y_1\\
                y_2\\
            \end{bmatrix}
            }.
        \end{equation}
    Since we have \(p_{C_1\mid \cdot}+p_{C_2\mid \cdot}+p_{C_D\mid \cdot}=1\), these three vectors satisfy \(\tilde{\mathbf{u}}_1+\tilde{\mathbf{u}}_2+\tilde{\mathbf{u}}_3=0\). Therefore it's unneccessary to write \(\tilde{\mathbf{u}}_3\) in the equation, as it is linearly dependent on \(\tilde{\mathbf{u}}_1,\tilde{\mathbf{u}}_2\). A solution of the equation is:
        \begin{equation}
            \footnotesize{
            \begin{bmatrix}
                \alpha_1\\
                \alpha_2\\
                \gamma\\
            \end{bmatrix}
            =
            \begin{bmatrix}
                0\\
                1\\
                -2\\
            \end{bmatrix},
            \begin{bmatrix}
                y_1\\
                y_2\\
            \end{bmatrix}
            =
            \begin{bmatrix}
                -5\\
                -2.5\\
            \end{bmatrix},
            \begin{bmatrix}
                \mathbf{s}_{C_1} & \mathbf{s}_{C_2}
            \end{bmatrix}
            =
            \begin{bmatrix}
               0.7 & 0.2\\
               0.4 & 0.4\\
               0.1 & 0.6\\
               0.6 & 0.2\\
               0.4 & 0.2\\
               0.2 & 0.2\\
               0.8 & 0\\
               0.5 & 0.2\\
               0.3 & 0.2\\
            \end{bmatrix}
            },
            \label{Eqt_FixPayoffStrategy}
        \end{equation}
    which indicates that player \(2\)'s payoff will be fixed to \(2\) (\(\bar{u}_2-2=0\)), provided she uses the strategy in Eq.(\ref{Eqt_FixPayoffStrategy}). It is enough to describe a Markov strategy by two vectors \(\mathbf{s}_{C_1},\mathbf{s}_{C_2}\), because the third vector \(\mathbf{s}_D\) is given by \(\mathbf{1}-\mathbf{s}_{C_1}-\mathbf{s}_{C_2}\). By using this strategy, player \(1\) unilaterally pins her opponent's payoff.

    Another solution is:
        \begin{equation}
            \footnotesize{
            \begin{bmatrix}
                \alpha_1\\
                \alpha_2\\
                \gamma\\
            \end{bmatrix}
            =
            \begin{bmatrix}
                1\\
                -1\\
                0\\
            \end{bmatrix},
            \begin{bmatrix}
                y_1\\
                y_2\\
            \end{bmatrix}
            =
            \begin{bmatrix}
                10\\
                5\\
            \end{bmatrix},
            \begin{bmatrix}
                \mathbf{s}_{C_1} & \mathbf{s}_{C_2}
            \end{bmatrix}
            =
            \begin{bmatrix}
               1 & 0\\
               0.5 & 0.4\\
               0.2 & 0.2\\
               0.7 & 0.2\\
               0 & 1\\
               0.1 & 0\\
               0.6 & 0.2\\
               0.3 & 0.2\\
               0 & 0\\
            \end{bmatrix}
            },
            \label{Eqt_FollowPayoffStrategy}
        \end{equation}
    which indicates that when she uses the strategy in Eq.(\ref{Eqt_FollowPayoffStrategy}), player \(1\)'s payoff will always be the same as her opponent's payoff. Figure \ref{Fig_TwoPlayerThreeAction} is the  numerical simulation of how the two ruling strategies derived in this example control the payoff pairs of the game. 
    \label{Emp_TwoPlayerThreeAction}
\end{example}

\begin{figure*}
    \centering
    \includegraphics[width=0.7\textwidth]{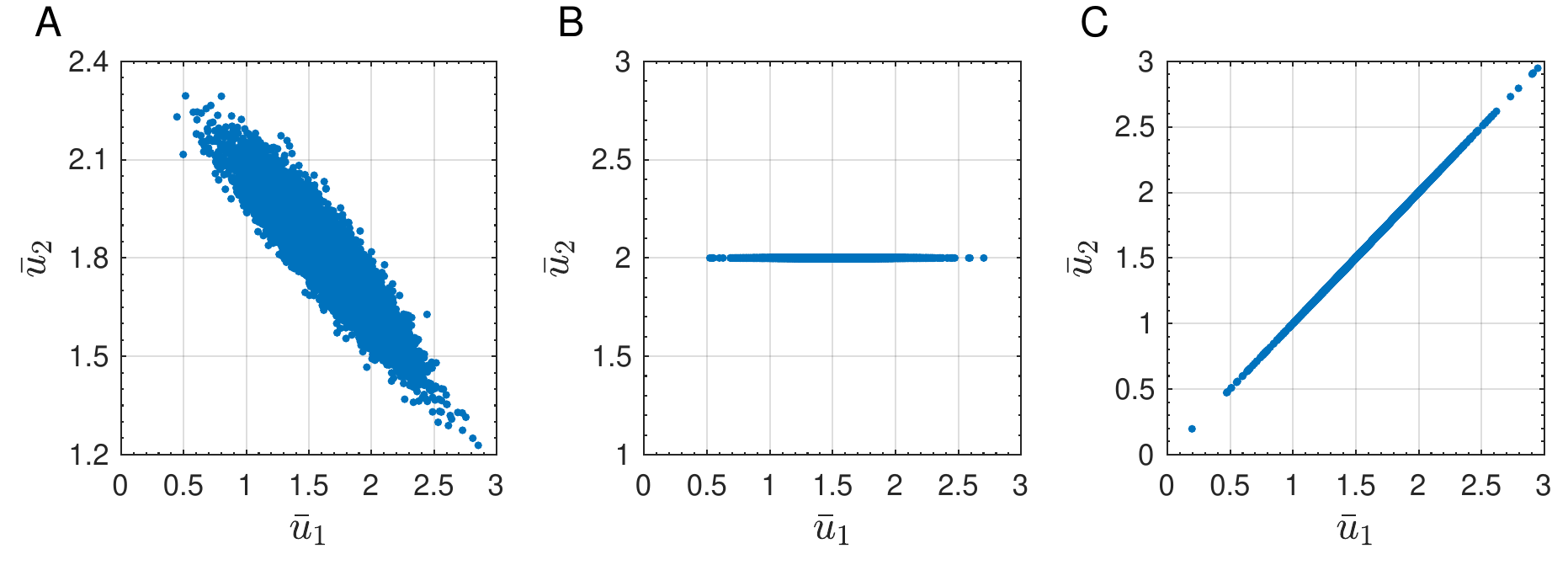}
    \caption{
        \textbf{The payoff control in a two-player three-action game.} Consider the game in Example \ref{Emp_TwoPlayerThreeAction}. 
        (\textbf{A}) Player \(1\) adopts a non-ruling strategy,  \(\mathbf{s}_{C_1}=[0.2,0.5,0.3,0.2,0.4,0.5,0.3,0.5,0.2],\mathbf{s}_{C_2}=[0.4,0.2,0.5,0.6,0.3,0,0.3,0.5,0.5]\).
        (\textbf{B}) Player \(1\) adopts a ruling strategy introduced in Eq.(\ref{Eqt_FixPayoffStrategy}). It sets player 2's payoff to a fixed value regardless of her strategy.
        (\textbf{C}) Player \(1\) adopts a ruling strategy introduced in Eq.(\ref{Eqt_FollowPayoffStrategy}), and unilaterally makes both players' payoff equal. In each panel, player 2' strategy is sampled for \(20000\) times.
    }
    \label{Fig_TwoPlayerThreeAction}
\end{figure*}

\begin{example}[Alliance rules, individuals do not]
    Consider a three-player infinitely repeated public goods game. Each round players choose between two actions: (i) action \(C\): contribute a cost \(3\) into the public pot; or (ii) action \(D\): contribute nothing; The total contribution in the public pot is multiplied by \(2\) and is then evenly distributed among all the three players, regardless of whether she contributed or not. The action profiles and corresponding payoff vectors are given by:
        \begin{equation}
            \footnotesize{
            \begin{blockarray}{cccc}
                \mathbf{a} & \mathbf{u}_1 & \mathbf{u}_2 & \mathbf{u}_3\\
                \begin{block}{c[c][c][c]}
                    CCC & 3 & 3 & 3 \\
                    CCD & 1 & 1 & 4 \\
                    CDC & 1 & 4 & 1 \\
                    CDD & -1 & 2 & 2 \\
                    DCC & 4 & 1 & 1 \\
                    DCD & 2 & -1 & 2 \\
                    DDC & 2 & 2 & -1\\
                    DDD & 0 & 0 & 0 \\
                \end{block}
            \end{blockarray}
            }.
        \end{equation}
    According to Theorem \ref{theorem:PDv_for_M-1strategy}, player \(1\) has two ruling vectors, 
    \begin{equation}
        \footnotesize{
            \begin{blockarray}{ccc}
                \mathbf{a} & \tilde{\mathbf{u}}_1 & \tilde{\mathbf{u}}_2\\
                \begin{block}{c[cc]}
                CCC & p_{C\mid CCC}-1 & p_{D\mid CCC}\\
                CCD & p_{C\mid CCD}-1 & p_{D\mid CCD}\\
                CDC & p_{C\mid CDC}-1 & p_{D\mid CDC}\\
                CDD & p_{C\mid CDD}-1 & p_{D\mid CDD}\\
                DCC & p_{C\mid DCC} & p_{D\mid DCC}-1\\
                DCD & p_{C\mid DCD} & p_{D\mid DCD}-1\\
                DDC & p_{C\mid DDC} & p_{D\mid DDC}-1\\
                DDD & p_{C\mid DDD} & p_{D\mid DDD}-1\\
                \end{block}
            \end{blockarray}
        }.
    \end{equation}
    But since \(\tilde{\mathbf{u}}_1 + \tilde{\mathbf{u}}_2=0\), it’s unneccessary to write \(\tilde{\mathbf{u}}_2\) in the equation, as it is linearly dependent on \(\tilde{\mathbf{u}}_1\).
    Suppose she wants to unilaterally pin player \(3\)'s payoff, to achieve this, the equation:
    \begin{equation}
        \footnotesize{
        \begin{blockarray}{ccc}
            \mathbf{a} & \mathbf{u}_3 & \mathbf{1}\\
            \begin{block}{c[cc]}
                CCC & 3 & 1\\
                CCD & 4 & 1\\
                CDC & 1 & 1\\
                CDD & 2 & 1\\
                DCC & 1 & 1\\
                DCD & 2 & 1\\
                DDC & -1 & 1\\
                DDD & 0 & 1\\
            \end{block}
        \end{blockarray}
        \begin{bmatrix}
            \alpha_3\\
            \gamma\\
        \end{bmatrix}
        =y_1
        \begin{blockarray}{c}
            \tilde{\mathbf{u}}_1 \\
            \begin{block}{[c]}
                p_{C\mid CCC}-1 \\
                p_{C\mid CCD}-1 \\
                p_{C\mid CDC}-1 \\
                p_{C\mid CDD}-1 \\
                p_{C\mid DCC} \\
                p_{C\mid DCD} \\
                p_{C\mid DDC} \\
                p_{C\mid DDD} \\
            \end{block}
        \end{blockarray}
        }.
    \end{equation} must have a non-zero solution. On the right side of the equation, the third and the fourth element are non-positive and the fifth and the sixth element are non-negative. On the left side of the equation, in row \(4\) and \(6\) the row vectors are the same. Therefore in the solution the dot product between the row vectors \([2,1]\) and the variable vector \([\alpha_3, \gamma]^T\) must be zero. However, this makes the dot product in row \(3\) and \(5\) non-zero, which leads to a contradiction. Therefore, player \(1\) alone cannot pin player \(3\)'s payoff. Since this is a symmetric game, neither can player \(2\).
    Nevertheless, when players \(1\) and \(2\) collaborate, they share a larger amount of ruling vectors. According to Theorem \ref{theorem:PDv_for_M-1Alliance}, we yield three linearly independent ruling vectors. The enlarged ruling space enables payoff control. Analogous calculation yields the following strategies (\(\mathbf{s}^1_C\) denotes player \(1\)'s strategy and player \(2\)'s strategy \(\mathbf{s}^2_C\)):
        \begin{equation}
            \footnotesize{
                    \bar{u}_1=1:
                \begin{bmatrix}
                    \mathbf{s}^1_C\\ 
                    \mathbf{s}^2_C\\
                \end{bmatrix}
                =\begin{bmatrix}
                    0.8 & 0.4 & 1 & 0.6 & 0.5 & 0.1 & 0.7 & 0.3\\
                    0.4 & 0.7 & 0 & 0.3 & 0.5 & 0.8 & 0.1 & 0.4\\
                \end{bmatrix}
            }.
            \label{Eqt_Collaborative_Ruling_Strategy_Protect}
        \end{equation}
        \begin{equation}
            \footnotesize{
                    \bar{u}_3=1:
                \begin{bmatrix}
                    \mathbf{s}^1_C\\ 
                    \mathbf{s}^2_C\\
                \end{bmatrix}
                =\begin{bmatrix}
                    0.6 & 0.7 & 0.4 & 0.3 & 0.4 & 0.5 & 0.2 & 0.1 \\
                    0.7 & 0.4 & 0.3 & 0.1 & 0.8 & 0.5 & 0.4 & 0.2 \\
                \end{bmatrix}
            }.
            \label{Eqt_Collaborative_Ruling_Strategy_Fix}
        \end{equation}
    In Eq.(\ref{Eqt_Collaborative_Ruling_Strategy_Protect}), the first collaborative ruling strategy set stabilizes player \(1\)'s payoff to \(1\). 
    The second collaborative ruling strategy set fixes player \(3\)'s payoff to \(1\). 
    Therefore, forming an alliance protects a member's payoff from outsiders or control the payoff of outsiders.
    Figure \ref{Fig_AllianceRules} is a numerical simulation of how the two collaborative ruling strategy sets derived in this example control the payoffs of the game.
    \label{Exp_AllianceRules_IndividualsDont}
\end{example}

\begin{figure*}
    \centering
    \includegraphics[width=0.7\textwidth]{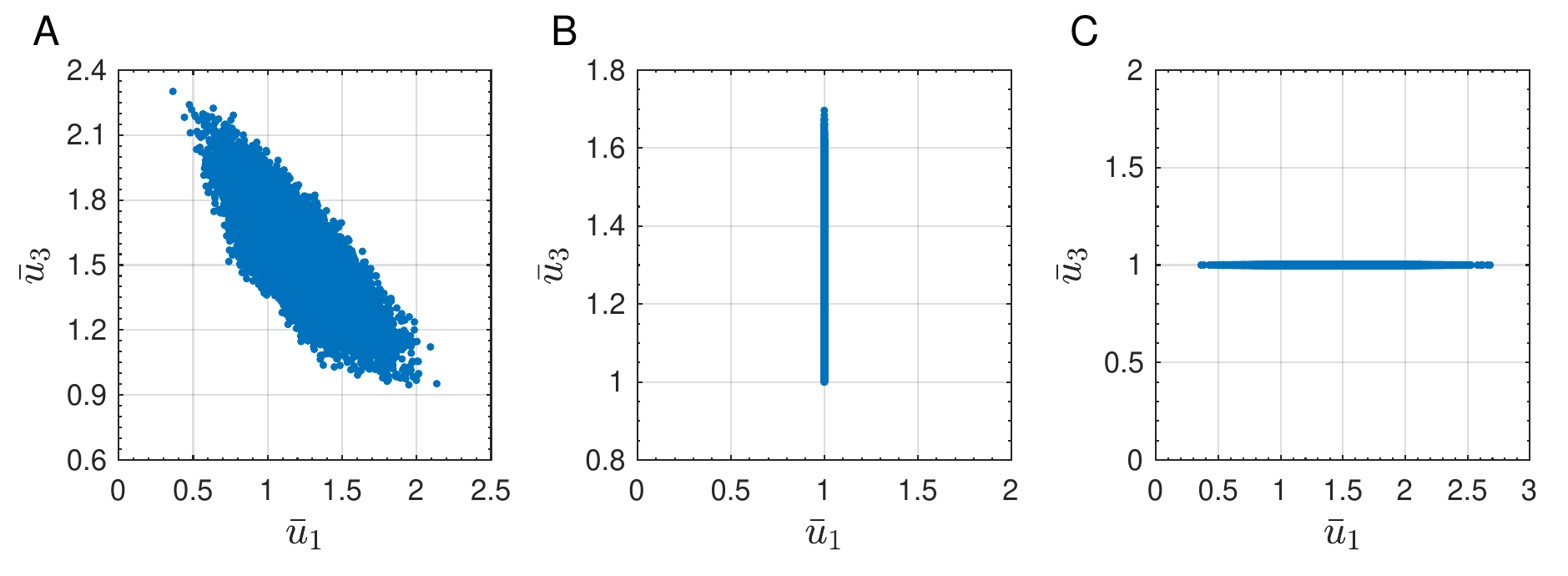}
    \caption{
        \textbf{The payoff control by alliance.} Consider the game in Example \ref{Exp_AllianceRules_IndividualsDont}. Player \(1\) and player \(2\) form an alliance against player \(3\). 
         (\textbf{A}) Player \(1\) and player \(2\) adopt non-collaborative ruling strategies, \(\mathbf{s}^1_C=[0.2,0.9,0.7,0.5,0.3,0.1,0.8,1],\mathbf{s}^2_C=[0.1,0.6,0,0.7,0.8,0,0.8,0.3]\).
        (\textbf{B}) Player \(1\) and player \(2\) adopt collaborative ruling strategies introduced in Eq.(\ref{Eqt_Collaborative_Ruling_Strategy_Protect}). They set player \(1\)'s payoff to a fixed value of \(1\), regardless of player \(3\)'s strategy. (\textbf{C}) Player \(1\) and player \(2\) adopt collaborative ruling strategies introduced in Eq.(\ref{Eqt_Collaborative_Ruling_Strategy_Fix}). They collectively set player $3$'s payoff to \(1\). In each panel, player 3's strategy is sampled for \(20000\) times. 
    }
    \label{Fig_AllianceRules}
\end{figure*}

We've shown how to find ruling strategies and collaborative ruling strategy sets in repeated games with infinite expected number of rounds and \(\delta\)-repeated games. 
The following theorem indicates that in a two-player two-action repeated game, ruling vectors for strict Markov strategies only exist in the above two cases.
If the player uses the same mixed action under every history, it indicates that the player has zero memory capacity and is unable to distinguish different histories. The strategy used is a degenerate case of Markov strategy. We term it \textit{memory-zero strategy}, and the rest of Markov strategies \textit{strict Markov strategies}.

\begin{theorem}[Existence theorem of ruling vectors for strict Markov strategies]
    In a two-player two-action repeated game, for every strict Markov strategy there exists a ruling vector if and only if the game satisfies one of the following conditions:
        \begin{itemize}
            \item The game is a \(\delta\)-repeated game;
            \item The expected number of rounds of the game is infinite.
        \end{itemize}
    \label{Thm_Existence_of_PDvecs}
\end{theorem}

\begin{proof} 
    We present a sketch of our proof.
    The existence of ruling vectors in the two cases is shown in Theorem \ref{theorem:PDv_for_M-1strategy} and \ref{theorem:PDv_for_M-1Alliance}. 
    Thus we mainly focus on proving the non-existence of ruling vectors in generalized games with finite expected rounds and are not \(\delta\)-repeated games.
    Suppose that we want to find a ruling vector for player \(1\).
    In a two-player two-action game, the inner product between a ruling vector and the probability distribution of that vector should always be zero.
    To prove non-existence of ruling vectors in a \(2\times 2\) generalized game, we show that under other conditions, for each vector, there exists a Markov strategy for the other player to use such that the inner product with the average distribution is not zero.
    We focus on such a strategy that satisfy the following conditions: (i) It is a Markov strategy. (ii) The average distribution \(\bar{\mathbf{v}}\) always exist. In this case, strategy \(\mathbf{s}_1,\mathbf{s}_2\) imply a markov matrix \(\mathbf{M}\) with stationary vector \(\mathbf{v}_{inv}\) \cite{Hauert1997}.

    First, for games with finite expected number of rounds, we proved that for matrix series \(\overline{\mathbf{M}}\) such that:
        \begin{equation}
            \begin{split}
                \bar{\mathbf{v}}
                &= \mathbf{v}^1\lim_{t\to\infty}\frac{p(1)\mathbf{I} + p(2)\mathbf{M} + \dotsb + p(t)\mathbf{M}^{t-1}} {p(1) + p(2) + \dotsb + p(t)}\\
                &= \mathbf{v}^1\overline{\mathbf{M}}
            \end{split}
        \end{equation}
    there exists an inverse \(\overline{\mathbf{M}}^{-1}=b_0I+b_1M + b_2 M^2 + b_3 M^3\) which \(b_0+b_1+b_2+b_3=1\) such that \(\mathbf{v}^1=\bar{\mathbf{v}}\overline{\mathbf{M}}^{-1}\). 
    Therefore, each entry in vector \(\mathbf{v}^1\) implies the result of an inner product between the column vectors of the matrix \(\overline{\mathbf{M}}^{-1}\), and the average distribution \(\bar{\mathbf{v}}\). 
    Every vector in \(\mathbb{R}^4\) is a linear combination of the \(4\) column vectors in \(\overline{\mathbf{M}}^{-1}\).
    Therefore, its inner product with \(\bar{\mathbf{v}}\) can be calculated by the same linear combination of entries of \(\mathbf{v}^1\).
    Each entry in the ruling vector should be determined by player \(1\)'s strategy, thus a valid ruling vector should eliminate parameters related to player \(2\)'s strategy to stay invariant when player \(2\)'s strategy changes.
    Also, the inner product between \(\bar{\mathbf{v}}\) and a valid ruling vector should be zero. 
    According to these conditions, We prove that only when \(b_2=b_3=0\) do ruling vectors exist. Therefore we have \(b_0+b_1=1\). Only \(\delta\)-repeated games satisfies such property, therefore finishing the proof. 

\end{proof}

This theorem presents the conditions for the existence of ruling vectors, indicating that strict Markov ruling strategies only exist in games with infinite expected rounds, and \(\delta\)-repeated games.
This theorem also suggests that the existence of ruling vectors is entirely determined by the continuation probability \(c\).
Nevertheless, it doesn't imply that ruling strategies derived from the two cases are totally not effective in games with a different \(c\). 
\cite{HilbeTor2014} proved that if the game is played sufficiently many rounds, the rule enforced will be a strict linear rule.

\section{DISCUSSION AND CONCLUSIONS}
\label{Sec_Discussion}

Previous studies in repeated games focused on the evolutionary stability of strategies. 
The issues of payoff control have received much less attention.
We defined a new class of strategies, namely ruling strategies, which are able to unilaterally set a linear payoff rule between the focal and other players.
Instead of adopting the previous determinant method \cite{Press2012}, we focused on the algebraic structure of ruling vectors. 
Firstly, we showed that payoffs can be given by the inner product between the action distribution and payoff vector.
Secondly, we defined ruling vectors and proved that ruling vectors yield a linear space, and so do payoff vectors. 
Finally, we proved that an overlap between the ruling space and the linear span of payoff vectors leads to payoff control.
With the aid of the algebraic perspective, we have provided a novel algorithm to find a ruling strategy and have shown that the existence of ruling vectors is only dependent on continuation probability, but not on payoffs.

Our algorithm overcomes the curse of dimensionality.
For instance, as the number of players increases, by the determinant method \cite{Press2012}, the size of the transition matrix increases exponentially, whereas in our algorithm, the size of the equation grows linearly.
Our method also facilitates the search for ruling strategies in games with arbitrary actions, and with asymmetric payoff vectors. Furthermore, our method allows finding ruling strategies for an alliance, which is typically challenging by the determinant framework \cite{Press2012}.

Our theory could be applied to games on networks \cite{SuAMing2019}, games with continuation probability dependent on the state, and even to stochastic games. Therefore, it opens an avenue to theoretically tackle payoff control problems.

\printbibliography[title = {REFERENCES}]

\end{document}